\documentclass[11pt,letterpaper]{amsart}
\usepackage{amsmath,amssymb,amsthm}
\usepackage{enumerate}
\usepackage{multirow}

\DeclareMathOperator{\Trace}{Tr}
\DeclareMathOperator{\Norm}{N}
\DeclareMathOperator{\Gal}{Gal}

\renewcommand{\phi}{\varphi}

\newcommand{\F}{{\mathbb F}}
\newcommand{\C}{{\mathbb C}}
\newcommand{\Q}{{\mathbb Q}}
\newcommand{\Z}{{\mathbb Z}}
\newcommand{\Fp}{\F_p}
\newcommand{\carmul}[1]{\widehat{\grmul{#1}}}

\newcommand{\card}[1]{\left|{#1}\right|}
\newcommand{\sums}[1]{\sum_{\substack{#1}}}
\newcommand{\mins}[1]{\min_{\substack{#1}}}

\newcommand\Gauss[1]{{\tau}_{#1}}
\newcommand\Weil[2]{{\rm W}_{#1,#2}}
\newcommand\Roots{R}

\newcommand\Val[2]{{\rm V}_{#1,#2}}
\newcommand\valuation{{\rm val}_{p}}
\newcommand\twovaluation{{\rm val}_{2}}

\newcommand\grmul[1]{{#1}^{\times}}

\newcommand\cardunits[1]{\card{\grmul{#1}}}

\newcommand{\norm}[2]{\Norm_{#1/#2}}
\newcommand{\trace}[2]{\Trace_{#1/#2}}

\newtheorem{theorem}{Theorem}[section]
\newtheorem{proposition}[theorem]{Proposition}
\newtheorem{lemma}[theorem]{Lemma}
\newtheorem{corollary}[theorem]{Corollary}
\newtheorem{conjecture}[theorem]{Conjecture}
\theoremstyle{definition}
\newtheorem{remark}[theorem]{Remark}

\title{Cyclotomy of Weil Sums of Binomials}

\author{Yves Aubry}
\address{Institut de Math\'ematiques de Toulon, Universit\'e de Toulon, France and Institut de Math\'ematiques de Marseille, CNRS-UMR 7373, Aix-Marseille Universit\'e, France}

\author{Daniel J.~Katz}
\address{Department of Mathematics, California State University, Northridge, \: United States}

\author{Philippe Langevin}
\address{Institut de Math\'ematiques de Toulon, Universit\'e de Toulon, France}

\date{first version: 12 December 2013; this version: 02 April 2015}

\begin{document}

\begin{abstract}
The Weil sum $W_{K,d}(a)=\sum_{x \in K} \psi(x^d + a x)$ where $K$ is a finite field, $\psi$ is an additive character of $K$, $d$ is coprime to $|K^\times|$, and $a \in K^\times$ arises often in number-theoretic calculations, and in applications to finite geometry, cryptography, digital sequence design, and coding theory.
Researchers are especially interested in the case where $W_{K,d}(a)$ assumes three distinct values as $a$ runs through $K^\times$.
A Galois-theoretic approach, combined with $p$-divisibility results on Gauss sums, is used here to prove a variety of new results that constrain which fields $K$ and exponents $d$ support three-valued Weil sums, and restrict the values that such Weil sums may assume.
\end{abstract}

\maketitle
\section{Introduction}
Let $K$ be a finite field of characteristic $p$.
Let $\psi_K$ be the canonical additive character
of $K$, that is, $\psi_K( x )=\exp( 2i\pi\trace K\Fp(x)/p )$ where 
$\trace K{\Fp}$ is the absolute trace.
{\it Weil sums} with $\psi_K$ applied to binomials, that is, sums of the form $\sum_{x \in K} \psi_K(b x^j + c x^k)$, have been studied extensively from the early twentieth century to present \cite{Kloosterman,Mordell,Vinogradow,Davenport-Heilbronn,Akulinicev,Karatsuba,Carlitz-1978,Carlitz-1979,Lachaud-Wolfmann,Katz-Livne,Coulter,Cochrane-Pinner-2003,Cochrane-Pinner-2011}.
We are interested in such sums when $j$ and $k$ are coprime to $\cardunits{K}$, in which case we reparameterize them to obtain sums of the form
\begin{equation}\label{James}
\Weil Kd(a) = \sum_{x \in K} \psi_K(x^d + a x)
\end{equation}
with $\gcd(d,\cardunits{K})=1$ and $a\in K$.
This definition will remain in force throughout the paper, and we shall always insist that $\gcd(d,\cardunits{K})=1$ whenever we write $\Weil Kd$.
The sums $\Weil Kd(a)$ are always real algebraic integers \cite[Theorem 3.1(a)]{Helleseth}, and furthermore, are all rational integers if and only if $d \equiv 1 \pmod{p-1}$ \cite[Theorem 4.2]{Helleseth}.
Apart from arising often in number-theoretic calculations, these sums are also the key to problems in finite geometry, cryptography, digital sequence design, and coding theory, as discussed in \cite[Appendix]{Katz}.

For a fixed $K$ and $d$, we consider $\Weil Kd(a)$ as a function of $a \in \grmul K$, and are interested in how many different values it assumes as $a$ 
runs through $\grmul K$.
$\Weil Kd(a)$ with $a=0$ is passed over, as it is the Weil sum of the monomial $x^d$, and since $x\mapsto x^d$ is a permutation of $K$, we always have $\Weil Kd(0)=0$.
We call $\{\Weil Kd(a): a \in \grmul K\}$ the {\it value set} of $\Weil Kd$, and say that $\Weil Kd$ is {\it $v$-valued} over $K$ to mean that this set is of cardinality $v$.

If $d \equiv p^j \pmod{\cardunits{K}}$ for some $j$, 
we say that $d$ is {\it degenerate over $K$},
because $\trace K\Fp(x^d+a x)=\trace K\Fp((1+a)x)$, and so the binomial 
effectively becomes zero (if $a=-1$) or a nonvanishing linear form (if $a\not=-1$).
Thus if $d$ is degenerate over $K$, one readily obtains for $a \in K$ that
\begin{equation}\label{Augustine}
\Weil Kd(a)=
\begin{cases} \card{K}  & \text{if $a=-1$,} \\ 
0 & \text{otherwise.} \end{cases}
\end{equation}
Helleseth \cite[Theorem 4.1]{Helleseth} shows that one always obtains a richer value set in the nondegenerate case.
\begin{theorem}[Helleseth, 1976] \label{Aaron}
If $d$ is nondegenerate over $K$, then $\Weil Kd(a)$ takes 
at least three values as $a$ runs through $\grmul K$.
\end{theorem}
Here we want to know when Weil sums of this form can be three-valued, and if so, what are the three values they may take.
We indicate all known infinite families of three-valued examples, arranged according to analogy, in Table \ref{Francis} below.
\begin{center}
\begin{table}[h!]\label{Francis}
\caption{Three-Valued Weil Sums}
\begin{center}
\begin{tabular}{c|c|c|c}
order of $K$ & $d$ (nondegenerate) & values of $W_{K,d}$ & reference \\
\hline
\hline
\multirow{2}{*}{$q=2^e$} & $d=2^i+1$ & \multirow{2}{*}{$0$, $\pm\sqrt{2^{\gcd(e,i)} q}$} & \multirow{2}{*}{\cite{Kasami-1966,Kasami-Lin-Peterson,Gold}} \\
 & $\twovaluation(i) \geq \twovaluation(e)$ & & \\
\hline 
$q=p^e$ & $d=\frac{1}{2}(p^{2 i}+1)$ & \multirow{2}{*}{$0$, $\pm\sqrt{p^{\gcd(e,i)} q}$} & \cite{Trachtenberg} ($e$ odd) \\
$p$ odd & $\twovaluation(i) \geq \twovaluation(e)$ & & \cite{Helleseth-Thesis, Helleseth} ($e$ even) \\
\hline
\multirow{2}{*}{$q=2^e$} & $d=2^{2 i}-2^i+1$, & \multirow{2}{*}{$0$, $\pm\sqrt{2^{\gcd(e,i)} q}$} & \multirow{2}{*}{\cite{Welch,Kasami-1971}} \\
 &  $\twovaluation(i) \geq \twovaluation(e)$ & & \\
\hline
$q=p^e$ & $d=p^{2 i}-p^i+1$ & \multirow{2}{*}{$0$, $\pm\sqrt{p^{\gcd(e,i)} q}$} & \cite{Trachtenberg} ($e$ odd) \\
$p$ odd & $\twovaluation(i) \geq \twovaluation(e)$ & & \cite{Helleseth-Thesis, Helleseth} ($e$ even) \\
\hline
$q=2^e$ & \multirow{2}{*}{$d=2^{e/2}+2^{(e+2)/4}+1$} & \multirow{2}{*}{$0$, $\pm 2 \sqrt{q}$} & \multirow{2}{*}{\cite{Cusick-Dobbertin}} \\
$\twovaluation(e)=1$ & & & \\
\hline
$q=2^e$ & \multirow{2}{*}{$d=2^{(e+2)/4}+3$} & \multirow{2}{*}{$0$, $\pm 2 \sqrt{q}$} & \multirow{2}{*}{\cite{Cusick-Dobbertin}} \\
$\twovaluation(e)=1$ & & & \\
\hline
$q=2^e$ & \multirow{2}{*}{$d=2^{(e-1)/2}+3$} & \multirow{2}{*}{$0$, $\pm\sqrt{2 q}$} & \multirow{2}{*}{\cite{Canteaut-Charpin-Dobbertin-1999,Canteaut-Charpin-Dobbertin-2000-Binary,Hollmann-Xiang}} \\
$e$ odd & & & \\
\hline
$q=3^e$ & \multirow{2}{*}{$d=2\cdot 3^{(e-1)/2}+1$} & \multirow{2}{*}{$0$, $\pm\sqrt{3 q}$} & \multirow{2}{*}{\cite{Dobbertin-Helleseth-Kumar-Martinsen}} \\
$e$ odd & & & \\
\hline
$q=2^e$ & $d=2^{2 i}+2^i-1$ & \multirow{2}{*}{$0$, $\pm\sqrt{2 q}$} & \multirow{2}{*}{\cite{Hollmann-Xiang,Hou}} \\
$e$ odd & $e\mid 4 i + 1$ & & \\
\hline
$q=3^e$ & $d=2\cdot 3^i+1$ & \multirow{2}{*}{$0$, $\pm\sqrt{3 q}$} & \multirow{2}{*}{\cite{Katz-Langevin-Proof}} \\
$e$ odd & $e\mid 4 i + 1$ & &
\end{tabular}
\end{center}
\end{table}
\end{center}
In several entries, we make use of the {\it $p$-adic valuation} of an integer $a$, denoted $\valuation(a)$, which is the maximum $k$ such that $p^k \mid a$ (or $\infty$ if $a=0$).
We write ``nondegenerate'' in the column heading for $d$ values to impose the condition that $d$ be nondegenerate over $K$ throughout the table, so that, for example, we cannot have $i=0$ in the first four rows.
If $K$ has characteristic $p$ and $1/d$ is interpreted modulo $\cardunits{K}$, then $\Weil K {p d}$ and $\Weil K {1/d}$ take the same values as $\Weil Kd$ \cite[Theorem 3.1]{Helleseth}, so the table records representative $d$ modulo these equivalences.

First of all, note that all these value sets consist of three rational integers, one of which is $0$, with the other two being opposites of each other.
The first two properties are inevitable facts, as shown in \cite[Theorems 1.7, 1.9]{Katz}.
\begin{theorem}[Katz, 2012]\label{Natalie}
Let $K$ be a finite field of characteristic $p$.
If $\Weil Kd$ is three-valued for some exponent $d$, then $d\equiv 1 \pmod{p-1}$, and the values must be rational integers, one of which is zero.
\end{theorem}
Concerning the two nonzero values of a three-valued Weil sum, one must be positive and the other negative, since it is known that $\sum_{a \in \grmul K} \Weil Kd(a)^2=\left(\sum_{a \in \grmul K} \Weil Kd(a)\right)^2$.
(See Lemma \ref{Prudence} and Corollary \ref{George} below for details.)
However, it has not been proved that these values must have the same magnitude, although this is always what has been observed.
We say that a three-valued Weil sum $\Weil Kd$ is {\it symmetric} when the two nonzero values are opposites of each other.
If we assume that a three-valued Weil sum is symmetric, we can make further conclusions about the possible values.
\begin{proposition}\label{Herve}
If $K$ is the finite field of characteristic $p$ and order $q$, and if $\Weil Kd(a)$ is three-valued with values $0$ and $\pm A$, then $|A|=p^k$ for some positive integer $k$ with $\sqrt{q} < p^k < q$.
\end{proposition}
This follows easily from well-known facts, which are arranged in Section \ref{Gerald}, where the above proposition is proved as Proposition \ref{Michael}.

Our first main result shows that in many cases, $\Weil Kd$ cannot be symmetric three-valued.
\begin{theorem}\label{Matilda}
Let $K$ be a finite field, and suppose that $I$ and $J$ are subfields of $K$ with $[J:I]=2$, with $d$ degenerate over $I$ but not over $J$.
Then the set of values assumed by $\Weil Kd(a)$ as $a$  runs through $\grmul K$ is not of the form $\{-A, 0, +A\}$ for any $A$.
\end{theorem}
We prove this in Section \ref{Arthur}.
This means that a field obtained by a tower of quadratic extensions over a prime field can never support a symmetric three-valued sum.
\begin{corollary}\label{Julianna}
Let $K$ be a finite field of characteristic $p$, and suppose that $[K:\Fp]$ is a power of $2$.  
Then the set of values assumed by $\Weil Kd(a)$ as $a$  runs through $\grmul K$ is not of the form $\{-A, 0, +A\}$ for any $A$.
\end{corollary}
For if $\Weil Kd$ were three-valued, Theorem \ref{Natalie} and eq.\thinspace\eqref{Augustine} would make $d$ degenerate over $\Fp$ but not over $K$, and then as we proceed from $\F_p$ toward $K$ up the tower of quadratic extensions, we must find a step where $d$ passes from degenerate to nondegenerate.
This corollary generalizes a result of Calderbank-McGuire-Poonen-Rubinstein \cite[Theorem 3]{Calderbank-McGuire-Poonen-Rubinstein}.
Our proof is quite different from that of Calderbank et al., who used McEliece's Theorem from coding theory (a relative of Stickelberger's Theorem on the $p$-divisibility of Gauss sums) and a delicate calculation in additive number theory to obtain Corollary \ref{Julianna} in the case where $p=2$.
The proof for Theorem \ref{Matilda} in full generality given here is much more straightforward, and is the consequence of some useful observations about the $p$-adic valuation of Weil sums.
These observations come as a consequence of relations (explored in Section \ref{Thomas}) between Weil and Gauss sums over a field and sums of the same form over a subfield: the Gauss sums play a role since Weil sums can be written in terms of Gauss sums, and the Davenport-Hasse relation supplies the connection between Gauss sums over the field and Gauss sums over the subfield.

Note that if $K$ is a field of characteristic $p$ and order $q=p^e$, with $e$ not equal to a power of $2$, then we can set $i=2^{\twovaluation(e)}$ in the first four rows of Table \ref{Francis} to obtain a $d$ such that $W_{q,d}$ is three-valued.
On the other hand, Table \ref{Francis} furnishes no example of a three-valued $W_{K,d}$ with $[K:\F_p]$ a power of $2$.
(Recall that our table prohibits parameters which make $d$ degenerate, so we cannot have $i$ a multiple of $e$ in the first four rows.)
Helleseth conjectured \cite[Conjecture 5.2]{Helleseth} that for such fields there is no $d$ that makes the Weil sum $W_{K,d}$ three-valued.
\begin{conjecture}[Helleseth, 1976]\label{Wilbur}
Let $K$ be a finite field of characteristic $p$.
If $[K:\Fp]$ is a power of $2$, then $\Weil Kd$ is not three-valued.
\end{conjecture}
If it were proved that three-valued Weil sums must be symmetric, this would follow from Corollary \ref{Julianna}.
The $p=2$ and $3$ cases of Conjecture \ref{Wilbur} have been proved.
First, Feng \cite[Theorem 2]{Feng} showed that if $p=2$, one could strengthen the conclusion of Corollary \ref{Julianna} to say that the value set is not only non-symmetric, but entirely lacks the value $0$.
Then when Katz \cite[Theorem 1.9]{Katz} proved that a three-valued Weil sum must take the value $0$, Conjecture \ref{Wilbur} was established for $p=2$.
Further work of Katz \cite[Theorem 1.7]{Katz-Divisibility} shows that Conjecture \ref{Wilbur} is also true when $p=3$.

A symmetric three-valued Weil sum is called {\it preferred} if the magnitude of the nonzero values is as small as possible in view of Proposition \ref{Herve}, that is, if the nonzero values are $\pm\sqrt{p q}$ when $q$ is an odd power of $p$, or if the nonzero values are $\pm p\sqrt{q}$ when $q$ is an even power of $p$.
This terminology originates from digital sequence design, wherein smaller magnitude Weil sums of binomials correspond to smaller cross-correlation between a pair of maximal linear recursive sequences, which is desirable.
The known infinite families of preferred three-valued Weil sums can be deduced from Table \ref{Francis} above: the last seven rows furnish preferred Weil sums, and in the first four rows, one must have $\gcd(e,i)=1$ if $e$ is odd, or $\gcd(e,i)=2$ if $e$ is even.

Our second main result is a lower bound on the magnitude of the nonzero values of a symmetric three-valued Weil sum $\Weil Kd$.
This bound grows as the $2$-divisibility of the degree of $K$ over its prime field increases.
\begin{theorem}\label{Nancy}
Let $K$ be the finite field of characteristic $p$ and order $q$.
If $\twovaluation([K:\Fp])=s$ and $\Weil Kd$ is symmetric three-valued with values $0, \pm A$, then $|A| \geq p^{2^{s-1}} \sqrt{q}$.
\end{theorem}
We prove this in Section \ref{Benjamin}.
One consequence is that if the degree of $K$ over its prime field is a multiple of $4$, then $\Weil Kd$ cannot be preferred.
\begin{corollary}\label{Hubert}
Let $K$ be the finite field of characteristic $p$ and order $q$.
If $[K:\Fp]\equiv 0\pmod 4$, then the set of values assumed by $\Weil{K}{d}$ as $a$ runs through $\grmul K$ is not of the form $\{0,\pm p\sqrt{q}\}$.
\end{corollary}
This generalizes the result of Calderbank-McGuire \cite{Calderbank-McGuire}, who proved a conjecture of Sarwate and Pursley \cite[p.~603]{Sarwate-Pursley}, which is the special case of Corollary \ref{Hubert} where $p=2$.
Our proof technique for Theorem \ref{Nancy} in full generality is much simpler than the original proof of Calderbank-McGuire, as it obviates the need for McEliece's Theorem or Stickelberger's Theorem.

Our first two results give restrictions on the types of fields that support symmetric and preferred Weil sums.
Our third result shows that certain exponents $d$ of the polynomial in the Weil sum prevent the Weil sum from being three-valued at all.
\begin{theorem}\label{Priscilla}
Let $K$ be a finite field of characteristic $p$ with $[K:\Fp]$ even.  If $d$ is a power of $p$ modulo $\sqrt{\card{K}}-1$, then $\Weil Kd$ is not three-valued.
\end{theorem}
In other words, it is impossible for $\Weil Kd$ to be three-valued if $K$ is the quadratic extension of a field $F$ in which $d$ is degenerate.
We prove this in Section \ref{Clarence}.
Such an exponent $d$ is called a {\it Niho exponent}, since they were first studied by Niho in \cite{Niho}.
Theorem \ref{Priscilla} generalizes the result of Charpin \cite[Theorem 2]{Charpin}, who proved the $p=2$ case.
Some steps of Charpin's proof for characteristic two do not hold in odd characteristic, so new arguments are devised.

Finally, the techniques developed here can be used to simplify the proof that the values of a three-valued Weil sum must be rational integers, a result that appears above in Theorem \ref{Natalie}, and which originally appeared in \cite[Theorem 1.7]{Katz}.
The new proof is presented in Section \ref{Lawrence}.

Our proofs of all the above results make extensive use of Galois theory.
Since Weil sums connect calculations in finite fields to calculations in cyclotomic extensions of $\Q$, there are two realms, both cyclotomic, where Galois groups come into play.
On the one hand, there are Galois groups for finite fields, which act on the terms of the polynomial arguments of the characters in the Weil sums; this is explored in Section \ref{Lester}.
On the other hand, there are Galois groups for cyclotomic fields, which are applied to the values of the Weil sums; this is explored in Section \ref{Penelope}.
This dual Galois-theoretic approach has proved to be both powerful for obtaining new results, and at the same time, simplifies the proofs of previous results that we recapitulate.

We should note that Weil sums assuming four, five, or more values are also studied (see \cite[Theorems 2.2, 2.3, 4.8, 4.10, 4.11, 4.13]{Helleseth} for some examples), but we focus on the three-valued ones, as they are extremal in view of Theorem \ref{Aaron}.
It has been asked \cite[Problem 3.6]{Katz-Langevin-New} whether there is an analogue of Theorem \ref{Natalie} for four-valued Weil sums.
Four-valued Weil sums $W_{K,d}(a)$ are known that assume irrational values and do not assume the value $0$ for $a \in \grmul K$.
For example, if $K$ is the field with $5$ elements and $d=3$, then $W_{K,d}(a)$ assumes four distinct irrational values ($\pm\sqrt{5}$ and $(5\pm\sqrt{5})/2$) as $a$ runs through $\grmul K$.
Thus any analogue of Theorem \ref{Natalie} for four-valued sums would need to be significantly different from the original.

The organization of this paper is as follows: in Section \ref{Gerald}, we prove some preliminary results using the well-known methodology of power moments.
In Section \ref{Lester}, we explore the action of the Galois groups of finite fields on the terms inside the Weil sums.
In Section \ref{Thomas}, we look at the Fourier transform of the value set of our Weil sums, which is expressible in terms of Gauss sums, from which we deduce results about the $p$-adic valuation of Weil sum values.
In Section \ref{Penelope}, we explore the action of the Galois groups of cyclotomic fields on the values of the Weil sums.
In Sections \ref{Arthur}, \ref{Benjamin}, and \ref{Clarence}, we prove Theorems \ref{Matilda}, \ref{Nancy}, and \ref{Priscilla}, respectively.
In Section \ref{Lawrence}, we finish with our new simpler proof of the rationality of the values of three-valued Weil sums.

\section{Power Moments of Weil Sums}\label{Gerald}

In this section we state some of the basic results about Weil sums that will be useful later on.
These facts are proved using character sums known as power moments.
Recall the definition \eqref{James} of $\Weil Kd$, and our tacit insistence that $\gcd(d,\cardunits{K})=1$ whenever we write $\Weil Kd$.
The $m$th {\it power moment} of the Weil sum $\Weil Kd$ is the sum
\[
\sum_{a \in \grmul K} \Weil Kd(a)^m.
\]
The first few power moments can be calculated as straightforward character sums.
\begin{lemma}\label{Prudence}
Let $K$ be a finite field.  Then
\begin{enumerate}[{\indent \rm (i).}]
\item\label{Edgar} $\sum_{a \in \grmul K} \Weil Kd(a)=\card{K}$,
\item\label{Therese} $\sum_{a \in \grmul K} \Weil Kd(a)^2=\card{K}^2$, and
\item\label{Iris} $\sum_{a \in \grmul K} \Weil Kd(a)^3=\card{K}^2 \cdot \card{\Roots}$,
\end{enumerate}
where $\Roots$ is the set of roots of the polynomial $(x+1)^d-x^d-1$ in $K$.
\end{lemma}
\begin{proof}
See \cite[Proposition 3.1]{Katz}.
\end{proof}
\begin{corollary}\label{Rufus}
If $K$ is a finite field, and $d$ is nondegenerate over $K$, then $|\Weil Kd(a)| < \card{K}$ for all $a \in \grmul K$.
\end{corollary}
\begin{proof}
From Lemma \ref{Prudence}\eqref{Therese}, the only way to escape this conclusion would be to have $|\Weil Kd(b)|=\card{K}$ for some $b \in \grmul K$, and $\Weil Kd(a)=0$ for all other $a$, which would make the Weil sum two-valued, contrary to Theorem \ref{Aaron}.
\end{proof}
\begin{corollary}\label{George}
If $d$ is nondegenerate over $K$, then $\Weil Kd$ assumes at least one positive value and at least one negative value.
\end{corollary}
\begin{proof}
Recall that the Weil sum values are real algebraic integers \cite[Theorem 3.1(a)]{Helleseth}.
By Theorem \ref{Aaron}, we know that $\Weil Kd$ must assume at least two nonzero values.
If all the nonzero values it assumes were of the same sign, then $\left(\sum_{a \in \grmul K} \Weil Kd(a)\right)^2 > \sum_{a \in \grmul K} \Weil Kd(a)^2$, contradicting Lemma \ref{Prudence}\eqref{Edgar} and \eqref{Therese}.
\end{proof}
The following is an easy consequence of this power moment analysis, and provides the proof of Proposition \ref{Herve} in the Introduction.
\begin{proposition}\label{Michael}
If $K$ is the finite field of characteristic $p$ and order $q$, and if $\Weil Kd(a)$ is three-valued with values $0$ and $\pm A$, then $d \equiv 1 \pmod{p-1}$ and $|A|=p^k$ for some positive integer $k$.
If $\Roots$ denotes the set of roots of $(x+1)^d-x^d-1$ in $K$, then $\sqrt{q} < \sqrt{\card{\Roots} q} = |A| < q$.
\end{proposition}
\begin{proof}
By Theorem \ref{Natalie}, we must have $A \in \Z$ and $d \equiv 1 \pmod{p-1}$.
Let $N_A$ be the number of $a \in \grmul K$ with $\Weil Kd(a)=A$.
Since the other two values $\Weil Kd(a)$ assumes are $0$ and $-A$, we have $\sum_{a\in \grmul K} \Weil Kd(a) (\Weil Kd(a)+A) = 2 A^2 N_A$,
and by Lemma \ref{Prudence}\eqref{Edgar},\eqref{Therese}, this sum also equals $q^2+q A$, so that $N_A =(q^2+q A)/(2 A^2)$, and so $A$ can not be divisible by any prime other than $p$.
We know $|A|< q$ by Corollary \ref{Rufus}.

Similarly, $\sum_{a \in \grmul K} \Weil Kd(a) (\Weil Kd(a)^2-A^2)=0$, and by Lemma \ref{Prudence}\eqref{Edgar},\eqref{Iris} equals $q^2 \card{\Roots}-q A^2$, so $|A|=\sqrt{\card{\Roots} q}$.
Then note that $0,-1 \in \Roots$.  (This is clear for $p=2$, and for $p$ odd, note that $\gcd(d,q-1)=1$ forces $d$ to be odd.)
Thus $A \geq \sqrt{2 q}$.
\end{proof}
It will also be useful to consider a version of the first power moment of a Weil sum, but where we restrict the summation to a smaller subfield.
\begin{lemma}\label{Vladislav}
Let $K$ be a finite field and let $L$ be the quadratic extension of $K$.
Then 
\[
\sum_{a \in \grmul K} \Weil Ld(a)=\card{L}.
\]
\end{lemma}
\begin{proof}
Let $q=\card{K}$.  Since $\Weil Ld(0)=0$, we have
\begin{align*}
\sum_{a \in \grmul K} \Weil Ld(a)
& = \sum_{x \in L} \psi_L(x^d) \sum_{a \in K} \psi_K(a \trace LK(x)) \\
& = q \sums{x \in L \\ \trace LK(x)=0} \psi_L(x^d).
\end{align*}
If $x \in L$ with $\trace LK(x)=0$, then $x^q=-x$, so that $\trace LK(x^d)=x^{q d}+x^d=(-x)^d+x^d=0$.  (In odd characteristic, $\gcd(d,q-1)=1$ makes $d$ odd.)
Thus $\sum_{a \in \grmul K} \Weil Ld(a) = q \cdot \card{\{x \in L: \trace LK(x)=0\}}=q^2=\card{L}$.
\end{proof}

\section{Action of Galois Groups of Finite Fields}\label{Lester}

We begin this section by seeing that the automorphisms of a finite field $K$ act trivially with respect to the Weil sum $\Weil Kd(a)$.  As always $\Weil Kd(a)$ is as defined in \eqref{James}, and $\gcd(d,\cardunits{K})=1$ whenever we write $\Weil Kd$.
\begin{lemma}
Let $K$ be a finite field of characteristic $p$.
If $\sigma \in \Gal(K/\Fp)$, then $\Weil Kd(\sigma(a))=\Weil Kd(a)$.
\end{lemma}
\begin{proof}
Since Galois conjugates have the same trace, they have the same character value.  Thus $\Weil Kd(a) = \sum_{x \in K} \psi_K(\sigma(x^d+a x))$, and by reparameterizing with $y=\sigma(x)$, we have $\Weil Kd(a)=\sum_{y \in K} \psi_K(y^d+\sigma(a) y)=\Weil Kd(\sigma(a))$.
\end{proof}
The action of the Galois group also shows that some exponents give equivalent Weil sums. 
\begin{lemma}\label{Geoffrey}
Let $K$ be a finite field of characteristic $p$.
Then $\Weil Kd(a)=\Weil K{p^j d}(a)$ for any $a \in K$ and $j \in \Z$.
\end{lemma}
\begin{proof}
This follows immediately from the fact that $x^{p^{j} d}$ is a Galois conjugate of $x^d$, and so $\psi_K(x^{p^j d})=\psi_K(x^d)$.
\end{proof}
Now we use finite field automorphisms to prove a congruence between the Weil sum over a field and the Weil sum over its extensions.
\begin{lemma}\label{Chiara}
Let $K$ be a finite field of characteristic $p$, and let $L$ be an extension of $K$ with $[L:K]$ a power of a prime $\ell$ distinct from $p$.  Then for any $a \in K$, we have 
\[
\Weil Ld(a) \equiv \Weil Kd([L:K]^{1-1/d} a) \pmod{\ell},
\]
where $1/d$ indicates the multiplicative inverse of $d$ modulo $p-1$.
\end{lemma}
\begin{proof}
For $a \in K$, we have
\[
\Weil Ld(a) = \sum_{x \in K} \psi_K(\trace LK(x^d+a x)) + \sum_{x \in L\smallsetminus K} \psi_L(x^d+a x).
\]
The first sum equals $\sum_{x \in K} \psi_K([L:K](x^d+a x))$, and if we reparameterize with $w=[L:K]^{1/d} x$, then we see that this sum is $\Weil Kd([L:K]^{1-1/d} a)$.
For the second sum, the action of $\Gal(L/K)$ partitions $L\smallsetminus K$ into orbits of Galois conjugates whose sizes are positive powers of $\ell$.
For any $\sigma \in \Gal(L/K)$, we have $\psi_L(x^d+a x)=\psi_L(\sigma(x^d+a x))=\psi_L(\sigma (x)^d + a \sigma(x))$, so that the value of $\psi_L(x^d+a x)$ is constant on orbits, and thus the sum over $L\smallsetminus K$ is $\ell$ times a sum of algebraic integers.
\end{proof}
We then explore what this tells us in the case where $d$ is degenerate in the smaller field.
\begin{corollary}\label{David}
Let $K$ be a finite field of characteristic $p$, and let $L$ be an extension of $L$ with $[L:K]$ a power of a prime $\ell$ distinct from $p$.
Let $d$ be degenerate over $K$.
Then $\Weil Ld(-1) \equiv \card{K} \pmod{\ell}$ and $\Weil Ld(a)\equiv 0 \pmod{\ell}$ for every $a \in K\smallsetminus\{-1\}$.
\end{corollary}
\begin{proof}
Combine Lemma \ref{Chiara} with \eqref{Augustine}, and note that since $d$ is degenerate over $K$, we have $d \equiv 1 \pmod{p-1}$, so the factor of $[K:L]^{1-1/d}$ mentioned in Lemma \ref{Chiara} is equal to $1$. 
\end{proof}

\section{Gauss Sum and Valuation}\label{Thomas}

In this section, we explore the Fourier transform of the value set of the Weil sum, which is expressible in terms of Gauss sums.
This will enable us to prove some criteria about the $p$-divisibility of Weil sum values.

Throughout this section $K$ is a finite field of characteristic $p$ and order $q$ and, as always, we assume that $\gcd(d,q-1)=1$.
For any multiplicative character $\chi\in\carmul K$, 
we consider the Gauss sum
\[
\Gauss K(\chi) = \sum_{a \in \grmul K} \chi(a) \psi_K(a).
\]
By Fourier inversion,
if $a\in\grmul K$, we find that
\[
 \psi_K( a ) = \frac 1{q-1} \sum_{\chi\in\carmul K} \Gauss K(\chi)\bar\chi(a).
\]
Thus for $a\in\grmul K$,
\begin{align}
 \Weil Kd(a)  &= 1 + \frac 1{(q-1)^2}\sum_{b\in\grmul K}\sum_{\chi,\phi\in\carmul K} \Gauss K(\chi)\Gauss K(\phi)  \bar\chi^d(b) \bar\phi(a b) \label{decomposition} \\
      &= 1 + \frac 1{q-1}\sums{\chi,\phi \in \carmul K \\ \phi=\bar\chi^d} \Gauss K(\chi) \Gauss K (\phi) \bar\phi(a) \nonumber \\
      &= \frac q{q-1} + \frac 1{q-1} \sum_{\chi\not=1} \Gauss K (\chi) \Gauss K(\bar\chi^{d}) \chi^d(a) \nonumber.
\end{align}
If we denote by $t$ the inverse of $-d$ modulo $q-1$, the above formula 
shows that $q$ and the $\Gauss K (\chi) \Gauss K(\bar\chi^{d})$ are the Fourier coefficients of the mapping $a\mapsto \Weil Kd(a^t)$ from $\grmul K$ to $\C$, whence by Fourier inversion
\begin{equation}\label{inversion}
\sum_{a\in\grmul K} \Weil K d(a^t) \chi(a) = \begin{cases}
$q$ & \text{if $\chi=1$,}\\
\Gauss K (\chi) \Gauss K(\bar\chi^{d}) & \text{otherwise.}
\end{cases}
\end{equation}

Recall from the Introduction that for any nonzero integer $n$, 
the {\it $p$-adic valuation} of $n$,
written $\valuation(n)$, is the largest $k$ 
such that $p^k$ divides $n$, 
and we set $\valuation(0)=\infty$.
Then $\valuation(a b)=\valuation(a)+\valuation(b)$ and $\valuation(a+b) \geq \min\{\valuation(a),\valuation(b)\}$, which becomes an equality whenever $\valuation(a)\not=\valuation(b)$.
We can extend the definition to $\Q$, wherein $\valuation(a/b)=\valuation(a)-\valuation(b)$.
If $\zeta_p$ and $\zeta_{q-1}$ are, respectively, primitive $p$th and $(q-1)$th roots of unity over $\Q$,  we can further extend $\valuation$ to the field $\Q(\zeta_p,\zeta_{q-1})$ where the Gauss sums reside, while still retaining the relations given above concerning products and sums of elements.
In this last field, elements can have fractional valuations: for instance $\valuation(1-\zeta_p)=1/(p-1)$.

We introduce the useful notation

\[
\Val Kd = \min_{a \in \grmul K} \valuation( \Weil Kd(a) ).
\]
It is well known \cite{Langevin}, \cite[Section 6]{Langevin-Veron} that Stickelberger's congruence on Gauss sums can be used
to obtain the value of $\Val Kd$ but we do not need it to reach our goal.
\begin{lemma}
For $K$ a finite field of order $q$, and $d$ an integer coprime to $q-1$, we have
$$\Val Kd  = \mins{\chi \in \carmul K \\ \chi\not=1} \valuation( \Gauss K(\chi) \Gauss K(\bar\chi^d) ).$$
\end{lemma}
\begin{proof}
This is straightforward once we note that $\valuation(\chi(a))=0$ 
for any $\chi \in \carmul K$ and any $a \in \grmul K$, 
because $(q-1)\valuation(\chi(a))=\valuation(\chi(a)^{q-1})=\valuation(1)=0$.
Using the relation (\ref{decomposition}), one 
has $\Val Kd \ge \min_{\chi\not=1} \valuation( \Gauss K(\chi) \Gauss K(\bar\chi^d) )$, and the reverse inequality is obtained
by using the relation (\ref{inversion}), once we establish that $\min_{\chi\not=1} \valuation(\Gauss K(\chi) \Gauss K(\bar\chi^d) ) \leq \valuation(q)$.
This last fact follows because $\Gauss K(\bar\chi)=\chi(-1) \overline{\Gauss K(\chi)}$ and $|\Gauss K(\chi)|^2=q$ for any nontrivial multiplicative character $\chi$, and so $\prod_{\chi\not=1} \Gauss K(\chi) \Gauss K(\bar\chi^d) = \pm q^{q-2}$.
\end{proof}
\begin{corollary}
\label{hasse}
Let $L$ be a finite extension of $K$.
For a positive integer $d$,
\[
  \Val Ld \leq [L:K] \times \Val Kd 
\]
\end{corollary}
\begin{proof}
Denoting by $\norm LK$
the norm from $L$ over $K$, the Davenport-Hasse relation (see \cite{Davenport-Hasse}) tells us that if $\chi \in \carmul K$, we have
\[
- \Gauss L( \chi\circ\norm LK ) = (-\Gauss K(\chi) )^{[L:K]},
\]
and the set of lifted characters $\chi\circ\norm LK$ as $\chi$ runs through the nontrivial elements of $\carmul K$ is a subset of the nontrivial elements of $\carmul L$.
\end{proof}
The remaining results in this section are specific to quadratic extensions of finite fields, which are involved in our three main results (Theorems \ref{Matilda}, \ref{Nancy}, and \ref{Priscilla}).
\begin{lemma}\label{Frederick}
Let $K$ be a finite field, and let $L$ be the quadratic extension of $K$.
Let $d$ be degenerate over $K$, but not over $L$.
Let $Y$ be a set of representatives of cosets of $\grmul K$ in $\grmul L$.
Then for $a \in L$, we have
\[
\Weil Ld(a) = \card{K} (Z(a)-1),
\]
where $Z(a)$ is the number of $y \in Y$ such that $\trace LK(y^d+ay)=0$. 
\end{lemma}
\begin{proof}
If $K$ has characteristic $p$, then Lemma \ref{Geoffrey} allows us to replace $d$ with $p^j d$ for any $j$, so we may take $d \equiv 1 \pmod{\cardunits{K}}$ without loss of generality.  Then 
\begin{align*}
\Weil Ld(a) & = 1+ \sum_{y\in Y} \sum_{x \in \grmul K} \psi_L( (y^d + ay)x )\\
& = -\card{K} + \sum_{y \in Y} \sum_{x \in K} \psi_K(x \trace LK(y^d+a y)),
\end{align*}
since $\card{Y}=(\card{L}-1)/(\card{K}-1)=\card{K}+1$.  The sum over $x$ is $\card{K}$ when $\trace LK(y^d+a y)=0$; otherwise the sum is $0$.
\end{proof}
This calculation has immediate consequences for the $p$-adic valuation of Weil sum values.
\begin{corollary}\label{subfield}
Let $K$ be a finite field of characteristic $p$, and let $L$ be the quadratic extension of $K$.
Let $d$ be degenerate over $K$, but not over $L$.
Then
\[
\Val Ld = [K:\Fp],
\]
and furthermore, $\Weil Ld(a)=-|K|$ for some $a \in \grmul L$.
\end{corollary}
\begin{proof}
Let $Y$ and $Z(a)$ be as defined in Lemma \ref{Frederick}, which tells us that 
\[
\Weil Ld(a) = \card{K}(Z(a)-1),
\]
for each $a \in L$.
All these numbers have a valuation greater or equal to $[K:\Fp]$.
Since $d$ is not degenerate over $L$, $\Weil Ld(a)$ must be negative for some $a \in \grmul L$ by Corollary \ref{George}.  
The only way to make $\Weil Ld(a)$ negative is to have $Z(a)=0$, which makes $\Weil Ld(a)=-\card{K}$, and then the valuation of $\Weil Ld(a)$ is precisely $[K:\Fp]$.
\end{proof}
The calculation of Lemma \ref{Frederick} also gives a nonnegativity condition that will be useful in our proof of Theorem \ref{Priscilla}.
\begin{corollary}\label{William}
Let $K$ be a finite field, and let $L$ be the quadratic extension of $K$.
Let $d$ be degenerate over $K$.
Then $\Weil Ld(a) \geq 0$ for all $a\in K$.
\end{corollary}
\begin{proof}
We may take $d$ nondegenerate over $L$, since \eqref{Augustine} settles the degenerate case.
Let $a \in K$.
By Lemma \ref{Frederick}, it suffices to find some $y \in \grmul L$ such that $\trace LK(y^d+a y)=1$.
In characteristic $2$, take $y \in \grmul K$, so that $\trace LK(y^d+a y)=2(y^d+a y)=0$.
In odd characteristic, take $y \in L$ with $y^2\in K$ but $y\not\in K$.
Then $y$ and $-y$ are conjugates under the action of $\Gal(L/K)$, and so $\trace LK(y^d+a y)=(-y)^d+a(-y)+y^d+a y=0$.
\end{proof}

\section{Action of Galois Groups of Cyclotomic Fields}\label{Penelope}

Throughout this section, $\zeta_p$ denotes a primitive $p$th root of unity over $\Q$.
If $K$ is a field of characteristic $p$, then the Weil sum values $\Weil Kd(a)$ reside in $\Q(\zeta_p)$ by definition \eqref{James}.
First we see how Galois automorphisms permute the Weil sum values.
Recall that we always have $d$ invertible modulo $\cardunits{K}$ whenever we write the sum $\Weil Kd$.
\begin{lemma}\label{Nathan}
Let $K$ be a finite field of characteristic $p$.
If $\sigma$ is the element of $\Gal(\Q(\zeta_p)/\Q)$ with $\sigma(\zeta_p)=\zeta_p^j$, then $\sigma(\Weil Kd(a))=\Weil Kd(j^{1-(1/d)} a)$, where $1/d$ indicates the multiplicative inverse of $d$ modulo $p-1$.
\end{lemma}
\begin{proof}
This is \cite[Theorem 2.1(b)]{Katz}.
\end{proof}
This shows that if two Weil sum values are Galois conjugates over $\Q$, then they occur equally often.
\begin{corollary}\label{Genevieve}
Let $K$ be a finite field, and let $A$ and $B$ be values assumed by $\Weil Kd$.
If $A$ and $B$ are Galois conjugates over $\Q$, then the number of $a \in \grmul K$ such that $\Weil Kd(a)=A$ is equal to the number of $a \in \grmul K$ such that $\Weil Kd(a)=B$.
\end{corollary}
\begin{proof}
Let $\sigma \in \Gal(\Q(\zeta_p)/\Q)$ with $\sigma(A)=B$, and let $j \in \F_p^\times$ such that $\sigma(\zeta_p)=\zeta_p^j$.
By Lemma \ref{Nathan}, $\Weil Kd(a)=A$ precisely when $\Weil Kd(j^{1-1/d} a)=B$.
\end{proof}
Often the Weil sums lie in a proper subfield of $\Q(\zeta_p)$.
We give a criterion for determining when this happens.
\begin{lemma}\label{Vivian}
Let $K$ be a finite field of characteristic $p$.  Let $E$ be the extension of $\Q$ generated by all the values of $\Weil Kd(a)$ for $a \in \grmul K$.  
Let $m$ be the smallest divisor of $p-1$ such that $d \equiv 1 \pmod{(p-1)/m}$.
Then $E$ is the unique subfield of $\Q(\zeta_p)$ with $[E:\Q]=m$.
\end{lemma}
\begin{proof}
An arbitrary $\sigma \in \Gal(\Q(\zeta_p)/\Q)$ takes $\zeta_p$ to $\zeta_p^j$ for some $j \in \F_p^\times$.
So by Lemma \ref{Nathan}, we have
\begin{equation}\label{Simon}
\sigma^n(\Weil Kd(a))=\Weil Kd(j^{n(1-1/d)} a)
\end{equation}
for any $a \in \grmul K$ and $n \in \Z$.

Since $d \equiv 1 \pmod{(p-1)/m}$, we see that $j^{m(1-1/d)}=1$ for any $j \in \F_p^\times$.
Thus if $\sigma \in \Gal(\Q(\zeta_p)/\Q)$, then $\sigma^m$ fixes all the values of $\Weil Kd$.
So the subgroup of index $m$ in $\Gal(\Q(\zeta_p)/\Q)$ fixes all values in $E$, and so $[E:\Q]$ is a divisor of $m$.

Conversely, if we set $n=[E:\Q]$ and Fourier transform both sides of \eqref{Simon} with a multiplicative character $\chi \in \carmul K$, we obtain
\[
\sum_{a \in \grmul K} \Weil Kd(a) \chi(a) = \sum_{a \in \grmul K} \Weil Kd(j^{n(1-1/d)} a) \chi(a).
\]
The right hand side is ${\bar \chi}(j^{n(1-1/d)})$ times the left hand side.
The left hand side is nonzero, since it is either $q$ if $\chi$ is principal, or a product of Gauss sums involving nontrivial characters (use \eqref{inversion} with $\chi^t$ in place of $\chi$, where $t$ is the inverse of $-d$ modulo $q-1$).
Thus we must have $\chi(j^{n(1-1/d)})=1$ for all $j \in \F_p^\times$ and all $\chi \in \carmul K$, which forces $d \equiv 1 \pmod{(p-1)/n}$.
By the minimality of $m$, this means that $[E:\Q]=n \geq m$.
\end{proof}
\begin{remark}
Values of $\Weil Kd$ are always algebraic integers, so that if these lie in a field $E$, they actually lie in the ring of algebraic integers in $E$.
\end{remark}
\begin{remark}\label{Ramon}
In view of the previous remark, the special case of Lemma \ref{Vivian} when $m=1$ states that the values of $\Weil Kd(a)$ for $a \in \grmul K$ all lie in $\Z$ if and only if $d \equiv 1 \pmod{p-1}$.  This was proved in \cite[Theorem 4.2]{Helleseth}.
\end{remark}
The next result is reminiscent of the power moments of Section \ref{Gerald}.
We shall combine it with Lemma \ref{Nathan} in Corollary \ref{Elaine} below.
\begin{lemma}\label{Richard}
Let $K$ be a finite field.  For any $b \in K$ with $b\not=1$, we have
\[
\sum_{a \in \grmul K} \Weil Kd(a) \Weil Kd(b a)=0.
\]
\end{lemma}
\begin{proof}
Since $\Weil Kd(0)=0$, we may include the $a=0$ term in
\begin{align*}
\sum_{a \in \grmul K} \Weil Kd(a) \Weil Kd(b a)
& = \sum_{x,y \in K} \psi_K(x^d+y^d) \sum_{a \in K} \psi_K(a(x+by)) \\
& = \card{K} \sums{x,y \in K \\ x+b y=0} \psi_K(x^d+y^d) \\
& = \card{K} \sum_{y \in K} \psi_K(y^d(1+(-b)^d)),
\end{align*}
which vanishes because $y \mapsto y^d$ is a permutation of $K$, and $1+(-b)^d\not=0$ since $b\not=1$.
\end{proof}
Now we combine Lemmata \ref{Nathan} and \ref{Richard}.
\begin{corollary}\label{Elaine}
If $K$ is a finite field and $\sigma \in \Gal(\Q(\zeta_p)/\Q)$ permutes the values of $\Weil Kd$ nontrivially, then
\[
\sum_{a \in \grmul K} \Weil Kd(a) \sigma(\Weil Kd(a)) = 0.
\]
\end{corollary}
\begin{proof}
Lemma \ref{Nathan} furnishes an element $b$ such that $\sigma(\Weil Kd(a))=\Weil Kd(b a)$ for all $a \in \grmul K$, and clearly $b\not=1$, for otherwise $\sigma$ would fix each value taken by $\Weil Kd$.  Lemma \ref{Richard} finishes the proof.
\end{proof}

\section{Proof of Theorem \ref{Matilda}}\label{Arthur}

We have three fields $I \subseteq J \subseteq K$ with $[J:I]=2$.
Let $p$ be the characteristic of our fields.
As always, $\gcd(d,\cardunits{K})=1$.
We are given that $d$ is degenerate in $I$, but not in $J$.

We want to show that the value set of $\Weil Kd$ is not of the form $\{0,\pm A\}$.
Suppose the contrary.
By Proposition \ref{Michael}, $|A|$ must be an integral power of $p$ with $\sqrt{\card{K}} < |A| < \card{K}$, so then 
\begin{align*}
\Val Kd & = \valuation(A) \\
& > \valuation(\sqrt{\card{K}}) \\
& = \frac{1}{2} [K:\Fp].
\end{align*} 
On the other hand, by Corollary \ref{hasse} and Corollary \ref{subfield}, we get a contradiction because
\begin{align*}
\Val Kd 
& \leq [K:J] \times \Val Jd \\
& = [K:J] \times [I:\Fp] \\
& = \frac{1}{2} [K:\Fp].
\end{align*}

\section{Proof of Theorem \ref{Nancy}}\label{Benjamin}

We have $K$ a finite field of characteristic $p$ and order $q$ with $[K:\Fp]$ divisible by $2^s$.
As always, $\gcd(d,q-1)=1$.
We suppose that $\Weil Kd$ is symmetric three-valued with values $0$ and $\pm A$, and our goal is to show that $|A| \geq p^{2^{s-1}} \sqrt{q}$.

Note that $\F_{p^{2^s}} \subseteq K$.
Since $\Weil Kd$ is three-valued, $d$ is degenerate over $\Fp$ by Theorem \ref{Natalie}.
If $d$ were nondegenerate over $\F_{p^{2^s}}$, then there must be subfields $I$ and $J$ of $\F_{p^{2^s}}$ with $[J:I]=2$ and $d$ degenerate over $I$ but not over $J$.
Then Theorem  \ref{Matilda} tells us that $\Weil Kd$ is not symmetric three-valued, contrary to our hypothesis.

So $d$ is degenerate over $\F_{p^{2^s}}$, and thus every point of $\F_{p^{2^s}}$ is an element of the set $\Roots$ of roots of $(x+1)^d-x^d-1$.
Thus $\card{\Roots} \geq p^{2^s}$, so Proposition \ref{Michael} tells us that $|A| = \sqrt{\card{R} q} \geq p^{2^{s-1}} \sqrt{q}$.

\section{Proof of Theorem \ref{Priscilla}}\label{Clarence}

We have $L$ a finite field with $[L:\Fp]$ even, and $d$ is a power of $p$ modulo $\sqrt{\card{L}}-1$.
We want to show that $\Weil Ld$ is not three-valued.

Since we are considering $\Weil Ld$, the exponent $d$ is an invertible element modulo $\card{L}$.
If $d$ is degenerate over $L$, then $\Weil Ld$ is at most two-valued by \eqref{Augustine}, so we assume that $d$ is nondegenerate over $L$ henceforth.
The proof that $\Weil Ld$ is not three-valued when $L$ is of characteristic $2$ is given as \cite[Theorem 2]{Charpin}, so we assume that we are in odd characteristic henceforth.

Assume $\Weil Ld$ is three-valued to show a contradiction.
By Theorem \ref{Natalie} and Corollary \ref{George}, these three values are all in $\Z$, one of them is $0$, one is positive, and one is negative.
Let $K$ be the subfield of $L$ with $[L:K]=2$.
Then by Corollary \ref{William}, we know that $\Weil Ld(a) \geq 0$ for all $a \in K$.
Corollary \ref{David} shows that $\Weil Ld(-1)$ is odd, and that $\Weil Ld(a)$ is even for all other $a \in K$.
Since these are nonnegative, the positive value of $\Weil Ld$ must be $\Weil Ld(-1)$, and $\Weil Ld(a)=0$ for all other $a \in K$.
But Lemma \ref{Vladislav} tells us that $\sum_{a \in \grmul K} \Weil Ld(a)=\card{L}$, which forces $\Weil Ld(-1)=\card{L}$.
This contradicts Corollary \ref{Rufus}, since $\Weil Ld$ was assumed to be nondegenerate over $L$.

\section{New Proof of the Rationality of Three-Valued Weil Sums}\label{Lawrence}

We suppose that $\Weil Kd$ is three-valued, and we want to show that those three values lie in $\Z$.  As for the rest of Theorem \ref{Natalie}, the conclusion that $d\equiv 1 \pmod{p-1}$ will then follow immediately from Remark \ref{Ramon}, and the proof that one of the three values is $0$ is given in \cite[Theorem 5.2]{Katz}, which is not very difficult to follow.  The proof of rationality given here, while complex, is considerably easier than the original, given as \cite[Theorem 4.1]{Katz}.

Let $p$ and $q$ be respectively the characteristic and order of $K$, and so $\gcd(d,q-1)=1$.
Let $\zeta_p$ be a primitive $p$th root of unity over $\Q$.
Let $\Weil Kd(a)$ take the three distinct values $A$, $B$, and $C$, respectively, for $N_A$, $N_B$, and $N_C$ values of $a \in \grmul K$.
By Lemma \ref{Nathan}, the Galois group $\Gal(\Q(\zeta_p)/\Q)$ permutes $A$, $B$, and $C$.
The field $\Q(A,B,C)$ is a cyclic Galois extension of $\Q$ since it is contained in the cyclic extension $\Q(\zeta_p)$ of $\Q$.
Let $\sigma$ be a generator of $\Gal(\Q(A,B,C)/\Q)$.
There are three possible actions of $\sigma$ upon $\{A,B,C\}$: (i) $\sigma$ is the identity permutation, (ii) $\sigma$ acts transitively, or (iii) $\sigma$ permutes a pair of these elements, and fixes the third.
As $A$, $B$, and $C$ are algebraic integers, they lie in $\Z$ if and only if they lie in $\Q$, and this occurs precisely in Case (i).
So it suffices to show that Cases (ii) and (iii) are impossible.

In Case (ii), Corollary \ref{Genevieve} tells us that $N_A=N_B=N_C$, so they all equal $(q-1)/3$.
Then Lemma \ref{Prudence}\eqref{Edgar} shows that $N_A A+ N_B B + N_C C=q$, so that $A+B+C=3+\frac{3}{q-1}$.
As $A+B+C$ is fixed by $\sigma$, it lies in $\Q$, and is at the same time an algebraic integer, so it lies in $\Z$.
This means that $q-1\mid 3$, which forces $p=2$, in which case $\zeta_p=-1$, and so the values of $\Weil Kd$ lie in $\Z$, contradicting our supposition that $\sigma$ permutes them nontrivially.
So Case (ii) is impossible.

Henceforth, we suppose that we are in Case (iii).
Without loss of generality, we suppose that the generator $\sigma$ of $\Gal(\Q(A,B,C)/\Q)$ has $\sigma(A)=B$, $\sigma(B)=A$, and $\sigma(C)=C$.
Then $\sigma$ is of order $2$, and so $\Q(A,B,C)$ is a quadratic extension of $\Q$ lying in $\Q(\zeta_p)$.
There is no such thing if $p=2$ (since $\zeta_p=-1$, so $\Q(\zeta_p)=\Q$).
Otherwise, since $\Q(\zeta_p)$ is cyclic of degree $p-1$ over $\Q$, this means that $\Q(A,B,C)$ is the unique quadratic extension of $\Q$ contained in $\Q(\zeta_p)$.
In view of the values of the quadratic Gauss sums \cite{Gauss}, we know that this unique quadratic extension must be $\Q(\sqrt{p})$ if $p\equiv 1 \pmod{4}$, or $\Q(\sqrt{-p})$ if $p\equiv 3 \pmod{4}$.
But since $A$, $B$, and $C$ are real (see \cite[Theorem 3.1(a)]{Helleseth} or \cite[Theorem 2.1(c)]{Katz}), the latter case is impossible, so we must have $p \equiv 1 \pmod{4}$ and $\Q(A,B,C)=\Q(\sqrt{p})$.
Then $C \in \Z$, since it is an algebraic integer fixed by $\sigma$, and $A=a+b\sqrt{p}$ and $B=a-b\sqrt{p}$, for some $a, b$ with $2 a$, $2 b$, and $a+b \in \Z$, since this is the form of algebraic integers in $\Q(\sqrt{p})$, as shown in \cite[Chapter IV, Theorem 2.3]{Lang}.

Then Lemma \ref{Prudence}\eqref{Edgar},\eqref{Therese} tells us that
\begin{align}
N_A A + N_B B + N_C C & = q, \label{Abigail} \\
N_A A^2 + N_B B^2 + N_C C^2 & = q^2. \label{Barbara}
\end{align}
Also $\sum_{a \in \grmul K} \Weil Kd(a) \sigma(\Weil Kd(a))=0$ by Corollary \ref{Elaine}, so
\begin{equation}\label{Cecilia}
N_A A B + N_B B A + N_C C^2 = 0.
\end{equation}
By Corollary \ref{Genevieve}, we have $N_A=N_B$, and since $A=a+b\sqrt{p}$ and $B=a-b\sqrt{p}$, our three equations \eqref{Abigail}, \eqref{Barbara}, and \eqref{Cecilia} become
\begin{align*}
2 N_A a + N_C C & = q, \\
2 N_A (a^2+p b^2) + N_C C^2 & = q^2, \\
2 N_A (a^2-p b^2) + N_C C^2 & = 0,
\end{align*}
and this system is equivalent to the system
\begin{align}
2 N_A a + N_C C & = q, \label{Deborah} \\
4 N_A a^2 + 2 N_C C^2 & = q^2, \label{Eleanor} \\
4 N_A p b^2 & = q^2. \label{Felicity}
\end{align}
From \eqref{Felicity} we see that $p\mid N_A$.
Note that $C\not=0$, since otherwise \eqref{Deborah} and \eqref{Eleanor} imply that $N_A=1$, contradicting $p\mid N_A$.
If we subtract \eqref{Eleanor} from $2(a+C)$ times equation \eqref{Deborah}, we obtain
\[
2 (2 N_A + N_C) a C = q (2 a + 2 C - q),
\]
and since $N_A+N_B+N_C=q-1$, with $N_A=N_B$, this gives
\[
2 (q-1) a C = q (2 a + 2 C - q).
\]
Examine the $p$-adic valuation of each side of this equation to see that $\max\{\valuation(a),\valuation(C)\} \geq \valuation(q)$.
Then by Corollary \ref{Rufus}, we see that $|C| < q$, and since $C\not=0$, we must have $\valuation(C) < \valuation(q) \leq \valuation(a)$, so that $q\mid 2 a$.
If we reduce \eqref{Deborah} modulo $q$, we see that $q\mid N_C C$, but since $q\nmid C$, we have $p\mid N_C$.
Thus $p\mid N_A$ and $p\mid N_C$, and so $p\mid (2 N_A+N_C)=q-1$, which is absurd.
Thus Case (iii) is impossible, and the proof is complete.

\section*{Acknowledgements}

The second author was supported in part by a Research, Scholarship, and Creative Activity Award from California State University, Northridge.
The second author thanks Tor Helleseth for help with the history of these researches.

\end{document}